\documentclass[reqno,12pt]{amsart}
\usepackage{bbm}
\usepackage[all,pdf]{xy}
\usepackage{epsfig}
\usepackage{amsmath}
\usepackage{amssymb}
\usepackage{amscd}
\usepackage{graphicx}
\usepackage{color}

\usepackage{mathptmx}

\allowdisplaybreaks[4]

\makeatletter
\@namedef{subjclassname@2020}{%
	\textup{2020} Mathematics Subject Classification}
\makeatother
\usepackage[colorlinks=true]{hyperref}
\usepackage{amsfonts}

\usepackage[top=25mm, bottom=27mm, left=27mm, right=27mm]{geometry}

\newtheorem{thm}{Theorem}[section]
\newtheorem{lemma}[thm]{Lemma}

\newtheorem{prop}[thm]{Proposition}

\newtheorem{ques}[thm]{Question}
\newtheorem{cor}[thm]{Corollary}

\newtheorem{rem}[thm]{Remark}

\newcommand{\norm}[1]{\left\Vert #1\right\Vert}
\newcommand{\nnorm}[1]{\lvert\!|\!| #1|\!|\!\rvert}

\def \N {\mathbb N}
\def \C {\mathbb C}
\def \Z {\mathbb Z}
\def \R {\mathbb R}
\def \Q {\mathbb Q}
\def \E {\mathbb E}

\def\B {\mathcal B}

\def \X {\mathcal{X}}
\def \Y {\mathcal{Y}}

\def \ep {\epsilon}

\numberwithin{equation}{section}

\begin{document}

\title[]{Some ergodic theorems involving Omega function and their applications}

	\author[]{Rongzhong Xiao}		
	\address[Rongzhong Xiao]{School of Mathematical Sciences, University of Science and Technology of China, Hefei, Anhui, 230026, PR China \& Aix-Marseille Université, CNRS, Institut de Mathématiques de Marseille, Marseille, France}
	\email{xiaorz@mail.ustc.edu.cn}
	

\begin{abstract}
In this paper, we build some ergodic theorems involving function $\Omega$, where $\Omega(n)$ denotes the number of prime factors of a natural number $n$ counted with multiplicities. As a combinatorial application, it is shown that for any $k\in \N$ and every $A\subset \N$ with positive upper Banach density, there are $a,d\in \N$ such that $$a,a+d,\ldots,a+kd,a+\Omega(d)\in A.$$
\end{abstract}

\maketitle 
\section{Introduction}
Let $\Omega(n)$ denote the number of prime factors of a natural number $n$ counted with multiplicities. In mulplicative number theory, a central topic is to study the asymptotic distribution of the values of $\Omega(n)$. 

In 2022, Bergelson and Richter \cite{BR22} gave an asymptotic characterization of $\Omega(n)$ from a dynamical point of view. By {\em topological dynamical system}, we mean a pair $(X,T)$, where $X$ is a compact metric space and $T:X\to X$ is a homeomorphism. We say that $(X,T)$ is {\em uniquely ergodic} if there is only a $T$-invariant Borel probability measure on $X$. 
\begin{thm}\label{thm1-1}
	$($\cite[Theorem A]{BR22}$)$ Let $(X,T)$ be a uniquely ergodic topological dynamical system with unique $T$-invariant Borel probability measure $\mu$. Then for any $f\in C(X),x\in X$, 
	$$\lim_{N\to\infty}\frac{1}{N}\sum_{n=1}^{N}f(T^{\Omega(n)}x)=\int_{X}fd\mu.$$
\end{thm}
Later, Loyd \cite{L23} built an analogue of Theorem \ref{thm1-1} in the sense of norm convergence. By {\em measure preserving system}, we mean a tuple $(X,\X,\mu,T)$, where $(X,\X,\mu)$ is a Lebesgue space and $T:X\to X$ is an invertible measure preserving tansformation. We say that $(X,\X,\mu,T)$ is {\em ergodic} if for any $A\in \X$ with $\mu(A\Delta T^{-1}A)=0$, then $\mu(A)=0$ or $1$.
\begin{thm}\label{thm1-2}
	$($\cite[Theorem 2.5]{L23}$)$ Let $(X,\X,\mu,T)$ be an ergodic measure preserving system. Then for any $f\in L^{2}(\mu)$, 
	$$\lim_{N\to\infty}\norm{\frac{1}{N}\sum_{n=1}^{N}f(T^{\Omega(n)}x)-\int_{X}fd\mu}_{L^{2}(\mu)}=0.$$
\end{thm} 
In 2024, Charamaras \cite{D24} extended Loyd's result to double ergodic averages case.  
\begin{thm}\label{thm1-3}
	$($\cite[Corollary 1.33]{D24}$)$ Let $T,S$ be two invertible measure preserving transformations acting on Lebesgue space $(X,\X,\mu)$ such that $(X,\X,\mu,T)$ and $(X,\X,\mu,S)$ are ergodic. Then for any $f,g\in L^{2}(\mu)$, 
	$$\lim_{N\to\infty}\norm{\frac{1}{N}\sum_{n=1}^{N}f(T^{n}x)g(S^{\Omega(n)}x)-\int_{X}fd\mu\int_{X}gd\mu}_{L^{2}(\mu)}=0.$$
\end{thm}
For any $A\subset \Z^{k}$, we define $d^{*}(A)$ by letting 
$$d^{*}(A)=\sup_{{\Phi}}\limsup_{N\to\infty}\frac{|A\cap \Phi_N|}{|\Phi_N|},$$ where the supremum is taken over all F\o lner sequences\footnote{A F\o lner sequences of $\Z^k$ is a sequence $\{\Phi_N\}_{N\in \N}$ of finite subsets  of $\Z^k$ such that for each $h\in \Z^k$, $\displaystyle \lim_{N\to\infty}\frac{|\left(\Phi_N+h\right)\Delta\Phi_N|}{|\Phi_N|}=0$.} ${\Phi}=\{\Phi_N\}_{N\in \N}$ in $\Z^k$. If $d^{*}(A)>0$, we say that $A$ has {\em positive upper Banach density}.
As a combinatorial application of Theorem \ref{thm1-3}, Charamaras \cite[Corollary 1.37]{D24} showed that for any $E\subset \N$ with positive upper Banach density, there exist $m,n\in \N$ such that $m,m+n,m+\Omega(n)\in E$.

Motivated by the above results, in this paper, we consider the following ergodic averages:
\begin{equation}\label{eq1-1}
	\frac{1}{N}\sum_{n=1}^{N}w(n)\prod_{i=1}^{k}f_{i}(T_{i}^{P_{i}(n)}x)\cdot f_{k+1}(S^{\Omega(n)}x),
\end{equation}
where $S,T_1,\ldots,T_k$ is a family of invertible measure preserving tansformations acting on Lebesgue space $(X,\X,\mu)$, $P_1,\ldots,P_k\in \Z[n]$, $f_1,\ldots,f_{k+1}\in L^{\infty}(\mu)$, and $w:\N\to \C$ is a sequence.

When $T_1,\ldots,T_k$ generate a nilpotent group, and $f_{k+1}$ and $w$ are constant, the norm convergence of \eqref{eq1-1} was proved by Walsh in \cite{W12}.

Firstly, we extend Theorem \ref{thm1-1} to a weighted form. 
\begin{thm}\label{TA}
	Let $(X,\X,\mu,T)$ be a measure preserving system. Then for any $f\in L^{1}(\mu)$, there are a full measure subset $X_{f}$ of $X$ and $f^{*}\in L^{1}(\mu)$ such that for any $x\in X_{f}$, any uniquely ergodic topological dynamical system $(Y,S)$ with unique $S$-invariant Borel probability measure $\nu$, any $g\in C(Y)$, and any $y\in Y$, 
	$$\lim_{N\to\infty}\frac{1}{N}\sum_{n=1}^{N}f(T^{n}x)g(S^{\Omega(n)}y)=f^{*}(x)\int_{Y}gd\nu.$$
 \end{thm}
\begin{rem}
	a. Let $\mathbb{P}$ be the set of all prime numbers. Given $k\in \N$, let $P_1,\ldots,P_k\in \Z[n]$. Assume that $(X,\X,\mu,T)$ suffices the following property: There is  $\mathcal{P}\subset \mathbb{P}$ with positive relative density\footnote{We say that $\mathcal{P}\subset \mathbb{P}$ has positive relative density if $\displaystyle \lim_{N\to\infty}\frac{|\{1\le n\le N:n\in\mathcal{P}\}|}{|\{1\le n\le N:n\in\mathbb{P}\}|}>0.$} such that for any distinct $p,q\in \mathcal{P}\cup \{1\}$ and any $g_1,\ldots,g_{2k}\in L^{\infty}(\mu)$, the limit 
	$$\lim_{N\to\infty}\frac{1}{N}\sum_{n=1}^{N}\prod_{i=1}^{k}T^{P_{i}(pn)}g_i\cdot \prod_{j=1}^{k}T^{P_{j}(qn)}g_{k+j}$$ exists for $\mu$-a.e. $x\in X$. Then by the method used in the proof of Theorem \ref{TA}, Theorem \ref{thm4-1},  \cite[Theorem 16.10]{HK-book} and Theorem \ref{thm2-1}, we have that for any $f_1,\ldots,f_{k}\in L^{\infty}(\mu)$, there are a full measure subset $X_{f_1,\ldots,f_k}$ of $X$ and $f^{*}\in L^{\infty}(\mu)$ such that for any $x\in X_{f_1,\ldots,f_k}$, any uniquely ergodic topological dynamical system $(Y,S)$ with unique $S$-invariant Borel probability measure $\nu$, any $g\in C(Y)$, and any $y\in Y$, 
	$$\lim_{N\to\infty}\frac{1}{N}\sum_{n=1}^{N}\prod_{i=1}^{k}f_{i}(T^{P_{i}(n)}x)\cdot g(S^{\Omega(n)}y)=f^{*}(x)\int_{Y}gd\nu.$$
	
	b. By \cite[Theorem 1.2]{L23}, in Theorem \ref{TA}, when there is no continuous restriction for $g$, it may fail even for $\nu$-a.e. $y\in Y$.
\end{rem}
Next, we introduce some applications of Theorem \ref{TA}. After applying Theorem \ref{TA} to {\em rotations on torus}, we can get the following corollary.
\begin{cor}\label{corA}
	Let $\alpha>0,\beta\in \R$ with $\alpha+\beta\ge 1$. Let $(Y,S)$ be a uniquely ergodic topological dynamical system with unique $S$-invariant Borel probability measure $\nu$. Then for any $g\in C(Y)$ and any $y\in Y$, 
	$$\lim_{N\to\infty}\frac{1}{N}\sum_{n=1}^{N}g(S^{\Omega([\alpha n +\beta])}y)=\int_{Y}gd\nu,$$ where for any $x\in \R$, $[x]$ is the largest integer such that $0\le x-[x]<1$. 
\end{cor}
\begin{rem}
	a. Let $k\in \N$ and $0\le r<k$. By applying Corollary \ref{corA} to $(\Z_{k}, S)$, $1_{\{r\}}$ and $y=0$, where $S:x\mapsto x+1$, we have that $$\lim_{N\to\infty}\frac{|\{1\le n\le N:k|\left(\Omega([\alpha n +\beta])-r\right)\}|}{N}=\frac{1}{k}.$$
	
	b.  By applying Corollary \ref{corA} to $(\Z_{2}, S)$, $g:\Z_{2}\to \{1,-1\},0\mapsto 1,1\mapsto -1$ and $y=0$, where $S:x\mapsto x+1$, we have that $$\lim_{N\to\infty}\frac{1}{N}\sum_{n=1}^{N}\mathbf{\lambda}([\alpha n +\beta])=0,$$ where $\mathbf{\lambda}$ is the {Liouville function}. i.e. $\mathbf{\lambda}:\N\to \{1,-1\},n\mapsto (-1)^{\Omega(n)}$.
\end{rem}
We go on applying Theorem \ref{TA} to {\em unipotent affine transformations} (Cf. \cite[Page 67-69]{F81}) to get the following weighted ergodic theorem, which can be viewed as an analogue of Theorem \ref{thm2-1}.
\begin{cor}\label{corB}
	 Let $k\in \N$ and $\alpha\in \R\backslash \Q$. Let $m$ be the Harr measure on $\mathbb{T}$. Then for any $f\in L^{1}(m)$, there is $A_{f}\subset \R^{k}$ with zero Lebesgue measure such that for any $(x_0,\ldots,x_{k-1})\notin A_{f}$, any uniquely ergodic topological dynamical system $(Y,S)$ with unique $S$-invariant Borel probability measure $\nu$, any $g\in C(Y)$, and any $y\in Y$, 
	 $$\lim_{N\to\infty}\frac{1}{N}\sum_{n=1}^{N}f(\alpha n^{k}+x_{k-1}n^{k-1}+\cdots+x_{1}n+x_0)g(S^{\Omega(n)}y)=\int_{\mathbb{T}}fdm\int_{Y}gd\nu.$$
\end{cor}

Now, let us come back to \eqref{eq1-1}. For its some cases, we can build related mean ergodic theorems, which can be viewed as some extensions of Theorem \ref{thm1-3}. 
\begin{thm}\label{TB}
	Given $k\in \N$, let $S,T_1,\ldots,T_{k}$ be a family of invertible measure preserving transformations acting on Lebesgue space $(X,\X,\mu)$. Let $P_1,\ldots,P_k$ be pairwise independent\footnote{$P_1,\ldots,P_k$ are pairwise independent if for any $1\le i<j\le k$, there are no non-zero $x,y\in \Z$ such that $xP_i+yP_j$ is a constant.} integer coefficients polynomials and any one of them is not of the form $cn^d+b$. Assume that $T_1,\ldots,T_k$ are commuting. Then for any $f_1,\ldots,f_{k+1}\in L^{\infty}(\mu)$, the limit 
	$$\lim_{N\to\infty}\frac{1}{N}\sum_{n=1}^{N}\prod_{i=1}^{k}f_{i}(T_{i}^{P_{i}(n)}x)\cdot f_{k+1}(S^{\Omega(n)}x)$$ exists in $L^{2}(\mu)$.
\end{thm}
The reason why we make some restrictions on the polynomials $P_1,\ldots,P_k$ is to use the pronilfactors to control the $L^{2}$-norm of \eqref{eq1-1}. If $T_1=\cdots=T_k$, then we can leave out the restrictions for polynomials. That is,
\begin{thm}\label{TC}
	Given $k\in \N$, let $P_1,\ldots,P_k\in \Z[n]$. Let $T,S$ be two invertible measure preserving transformations acting on Lebesgue space $(X,\X,\mu)$. Then for any $f_1,\ldots,f_{k+1}\in L^{\infty}(\mu)$, the limit 
	$$\lim_{N\to\infty}\frac{1}{N}\sum_{n=1}^{N}\prod_{i=1}^{k}f_{i}(T^{P_{i}(n)}x)\cdot f_{k+1}(S^{\Omega(n)}x)$$ exists in $L^{2}(\mu)$.
\end{thm}

Lastly, we apply Theorem \ref{TB} and Theorem \ref{TC} to search for some additive structures in the sets with positive upper Banach density. Finally, we can get the following results.
\begin{prop}\label{PA}
	Given $k\in \N$, let $P_1,\ldots,P_k$ be pairwise independent integer coefficients polynomials with zero constant terms and any one of them is not of the form $cn^d$. Then for every $A\subset \N^{k+1}$ with positive upper Banach density, there are $a\in \N^{k+1},d\in \N$ such that $$a,a+P_{1}(d)\vec{e}_{1},\ldots,a+P_{k}(d)\vec{e}_{k},a+\Omega(d)\vec{e}_{k+1}\in A,$$ where $\{\vec{e}_1,\ldots,\vec{e}_{k+1}\}$ is the standard basis of $\R^{k+1}$.
\end{prop}
\begin{prop}\label{PB}
	Given $k\in \N$, let $P_1,\ldots,P_k\in \Z[n]$ with zero constant terms, then for every $A\subset \N^{2}$ with positive upper Banach density, there are $(x,y)\in \N^{2},d\in \N$ such that $$(x,y),(x+P_{1}(d),y),\ldots,(x+P_{k}(d),y),(x,y+\Omega(d))\in A.$$
\end{prop}
\begin{prop}\label{PC}
	Given $k\in \N$, let $P_1,\ldots,P_k\in \Z[n]$ with zero constant terms, then for every $A\subset \N$ with positive upper Banach density, there are $a,d\in \N$ such that $$a,a+P_{1}(d),\ldots,a+P_{k}(d),a+\Omega(d)\in A.$$
\end{prop}
Conventionally, to prove the above results, we use Furstenberg's corresponding principle to transfer them into some multiple recurrence results, which can be deduced from Theorem \ref{TB}, Theorem \ref{TC}, and the polynomial Szemer\'edi theorem \cite[Theorem A]{BL96}.
\subsection*{Organization of the paper} In Section 2, we recall some notions and results. In Section 3, we prove Theorem \ref{TA}. In Section 4, we show Corollary \ref{corA} and Corollary \ref{corB}. In Section 5, we prove Theorem \ref{TB} and Theorem \ref{TC}. In Section 6, we show Proposition \ref{PA} - \ref{PC}.  In Section 7, we list some questions. 
\subsection*{Acknowledgements} The author is supported by National Natural Science Foundation of China (12371196, 123B2007). The initial ideas of the paper first arose in an online seminar held at the winter of 2023. The author's thanks go to Zhengxing Lian, who organized this seminar.
\section{Preliminaries}
\subsection{Isomorphism and factors}
We say that measure preserving systems $(X,\X,\mu,T)$ and $(Y,\Y,\nu,S)$ are {\em isomorphic} if there exists an invertibe measure preserving transformation $\Phi:X_0\to Y_0$ with $\Phi \circ T=S\circ\Phi$, where $X_0$ is a $T$-invariant full measure subset of $X$ and $Y_0$ is an $S$-invariant full measure subset of $Y$.

A {\em factor} of a measure preserving system $(X,\X,\mu,T)$ is a $T$-invariant sub-$\sigma$-algebra of $\X$. 
\subsection{Conditional expectation and ergodic decomposition}
Given a Lebesgue space $(X,\X,\mu)$, let $\mathcal{Y}$ be a sub-$\sigma$-algebra of $\X$. For any $f\in L^{1}(X,\X,\mu)$, the {\em conditional expectation of $f$ with respect to $\Y$} is the function $\E_{\mu}(f|\Y)$, defined in $L^{1}(X,\Y,\mu)$, such that for any $A\in \Y$, $\int_{A}fd\mu=\int_{A}\E_{\mu}(f|\Y)d\mu$. Then there exists a unique $\Y$-measurable map $X\to \mathcal{M}(X,\X),x\mapsto \mu_x$, called the {\em disintegration of $\mu$ with respect to $\Y$}, under neglecting $\mu$-null sets such that for any $f\in L^{\infty}(X,\X,\mu)$, $\E_{\mu}(f|\Y)(x)=\int_{X}fd\mu_{x}$ for $\mu$-a.e. $x\in X$, where $\mathcal{M}(X,\X)$ is the collection of probability measures on $(X,\X)$, endowed with standard Borel structure.

Let $(X,\X,\mu,T)$ be a measure preserving system and $\mathcal{I}(T)$ be the sub-$\sigma$-algebra of $\X$, generated by all $T$-invariant sets. The disintegration of $\mu$ with respect to $\mathcal{I}(T)$, denoted by $X\to \mathcal{M}(X,\X),x\mapsto \mu_x$, is called the {\em ergodic decomposition of $\mu$ with respect to $T$}. Then for $\mu$-a.e. $x\in X$, $(X,\X,\mu_x,T)$ is an ergodic measure preserving system.
\subsection{Nilsequences}
	Let $k\in\N$. A {\em $k$-step nilmanifold} $X$ is a quotient space $G/\Gamma$, where $G$ is a $k$-step nilpotent Lie group and $\Gamma$ is a cocompact discrete subgroup of $G$. A {\em basic $k$-step nilsequence} is the sequence $\{f(a^{n}\cdot x)\}_{n\in \Z}$, where $f\in C(X),a\in G,x\in X$. A {\em $k$-step nilsequence} is a uniform limit of basic $k$-step nilsequences. Clearly, all $k$-step nilsequences is an invariant  algebra of $l^{\infty}(\Z)$ under translation. (For more details, see \cite[Chapter 11]{HK-book}.). By \cite[Proposition 11.13]{HK-book}, we have that for any nilsequence $\{b_n\}_{n\in \Z}$, the limit $\displaystyle \lim_{N\to\infty}\frac{1}{N}\sum_{n=1}^{N}b_{n}$ exists.
	
	For nilsequences, Bergelson and Richter \cite{BR22} built a disjointness result on it.
	\begin{thm}\label{thm2-1}
		$($\cite[Corollary 1.28 and Lemma 6.3]{BR22}$)$ Let $(X,T)$ be a uniquely ergodic topological dynamical system with unique $T$-invariant Borel probability measure $\mu$. Let $\{b_n\}_{n\in \Z}$ be a nilsequence. Then for any $f\in C(X),x\in X$, 
		$$\lim_{N\to\infty}\frac{1}{N}\sum_{n=1}^{N}b_{n}f(T^{\Omega(n)}x)=\int_{X}fd\mu\cdot \lim_{N\to\infty}\frac{1}{N}\sum_{n=1}^{N}b_{n}.$$
	\end{thm}
\subsection{Pronilfactors}
Fix a measure preserving system $(X,\X,\mu,T)$. For any $f\in L^{\infty}(\mu)$, we let
$\nnorm{f}_{1}=\norm{\E_{\mu}(f|\mathcal{I}(T))}_{L^{2}(\mu)}$.
Next, we define $\nnorm{\cdot}_{k}$ inductively. For each $k\ge 1$ and any $f\in L^{\infty}(\mu)$, we let 
$$\nnorm{f}_{k+1}=\left(\lim_{H\to\infty}\frac{1}{H}\sum_{h=0}^{H-1}\nnorm{f\cdot T^{h}\bar{f}}_{k}^{2^k}\right)^{\frac{1}{2^{k+1}}}.$$ 

By \cite[Theorem 9.7]{HK-book}, there exists a factor $\mathcal{Z}_{k}(T)$ of $(X,\X,\mu,T)$, called the {\em $k$-step factor}, such that for any $f\in L^{\infty}(\mu)$, $\nnorm f_{k+1}=0$ if and only if $\mathbb{E}(f|\mathcal{Z}_{k}(T))=0$. By \cite[Equation (8.15)]{HK-book}, we know that for any $k\in\N$, $\mathcal{Z}_{k}(T)\subset \mathcal{Z}_{k+1}(T)$. So, we can define the factor $\mathcal{Z}_{\infty}(T)$ of $(X,\X,\mu,T)$, called the {\em $\infty$-step factor}, by letting it be the smallest $\sigma$-algebra containing $\displaystyle\bigcup_{k= 1}^{\infty}\mathcal{Z}_{k}(T)$. $($For more details, see \cite[Chapter 8,9,16]{HK-book}.$)$.
\subsection{An othogonality criterion} The following othogonality criterion (Cf. 
\cite[Lemma 1]{D75}, \cite[Equation (3.1)]{K86}, \cite[Theorem 2]{BSZ13}) is an application of Tur\'an-Kubilius inequality.
\begin{lemma}\label{lem2-1}
	$($\cite[Lemma 2.14]{D24}$)$ Let $\{A(n)\}_{n\ge 1}$ be a bounded sequence in Hilbert space $\mathcal{H}$. If $\mathcal{P}\subset \mathbb{P}$ with positive relative density such that for any distinct $p,q\in \mathcal{P}$, 
	$$\lim_{N\to\infty}\frac{1}{N}\sum_{n=1}^{N}\langle A(pn),A(qn)\rangle=0,$$ then $$\lim_{N\to\infty}\norm{\frac{1}{N}\sum_{n=1}^{N}A(n)}=0.$$
\end{lemma}
\section{Proof Theorem \ref{TA}}
Firstly, let us recall Bourgain's double recurrence theorem. 
\begin{thm}\label{thm3-1}
	$($\cite[Main Theorem and Equation (2.3)]{B90}$)$ Let $(X,\X,\mu,T)$ be a measure preserving system. Fix two distinct non-zero integers $a$ and $b$. Fix $f,g\in L^{\infty}(\mu)$. If $f$ or $g$ is orthogonal to the closed subspace of $L^{2}(\mu)$ spanned by all eigenfunctions with respect to $T$, then $$\lim_{N\to\infty}\frac{1}{N}\sum_{n=1}^{N}T^{an}f\cdot T^{bn}g=0$$ for $\mu$-a.e. $x\in X$.
\end{thm}
Now, we prove Theorem \ref{TA}.
\begin{proof}[Proof of Theorem \ref{TA}]
	The whole proof is divided into three steps.
	
	\medskip
	
	\noindent\textbf{Step I. Reduction to ergodic case.}
	
	\medskip
	
	In this step, we show that to verify that Theorem \ref{TA} holds, it suffices to prove that Theorem \ref{TA} holds for all ergodic measure preserving systems.
	
	Suppose that Theorem \ref{TA} holds for all ergodic measure preserving systems. Next, we use contradiction argument to verify that Theorem \ref{TA} holds. 
	
	If Theorem \ref{TA} does not hold, then there is a measure preserving system $(X,\X,\mu,T)$ and $f\in L^{1}(\mu)$ such that there is a measurable set $X'\subset X$ of positive $\mu$-measure such that for each $x\in X'$, there are a uniquely ergodic topological dynamical system $(Y_{x},S_{x})$ with unique $S_{x}$-invariant Borel probability measure $\nu_{x}$, $g_x\in C(Y)$, and $y_x\in Y$ such that
	\begin{equation}\label{eq3-1}
		\limsup_{N\to\infty}\left|\frac{1}{N}\sum_{n=1}^{N}f(T^{n}x)g_{x}(S_{x}^{\Omega(n)}y_{x})-f^{*}(x)\int_{Y_{x}}g_{x}d\nu_{x}\right|>0, 
	\end{equation}
	where $$f^{*}(x)=\lim_{N\to\infty}\frac{1}{N}\sum_{n=1}^{N}f(T^{n}x).$$ Let $X\to \mathcal{M}(X),z\to \mu_{z}$ be the ergodic decomposition of $\mu$ with respect to $T$. Then there is $z\in X$ such that the following hold:
	\begin{itemize}
		\item[(1)] $(X,\X,\mu_{z},T)$ is ergodic;
		\item[(2)] $f\in L^{1}(\mu_z)$;
		\item[(3)] $\mu_{z}(X')>0$.
	\end{itemize}
	Then by the hypothesis and Birkhoff's ergodic theorem, we know that there is a measurable set $X''\subset X$ of full $\mu_{z}$-measure such that for each $x\in X''\cap X'$ such that  
	\begin{equation}\label{eq3-2}
		\lim_{N\to\infty}\left|\frac{1}{N}\sum_{n=1}^{N}f(T^{n}x)g_{x}(S_{x}^{\Omega(n)}y_{x})-\int_{X}fd\mu_{z}\int_{Y_{x}}g_{x}d\nu_{x}\right|=0
	\end{equation}
	and $\displaystyle \int_{X}fd\mu_{z}=f^{*}(x)$.
	Then \eqref{eq3-2} contradicts with \eqref{eq3-1}. So, Theorem \ref{TA} can be reduced to ergodic case.
	
	\medskip
	
	\noindent\textbf{Step II. Reduction to $L^{\infty}$-functions.}
	
	\medskip
	
	Based on \textbf{Step I}, we fix an ergodic $(X,\X,\mu,T)$. In this step, we show that to verify that Theorem \ref{TA} holds for $(X,\X,\mu,T)$, it suffices to prove that Theorem \ref{TA} holds for all $h\in L^{\infty}(\mu)$.
	
	Suppose that Theorem \ref{TA} holds for all $h\in L^{\infty}(\mu)$. Fix $f\in L^{1}(\mu)$. Then there is a sequence of functions $\{f_j\}_{j\ge 1}$ such that the following hold:
	\begin{itemize}
		\item[(1)] for each $j\ge 1$, $\norm{f-f_j}_{L^{1}(\mu)}<1/2^{j}$;
		\item[(2)] for each $j\ge 1$, there is a full measure subset $X_{j}$ of $X$ such that Theorem \ref{TA} holds for $f_j$.
	\end{itemize}
	Let 
	\begin{align*}
		& X_0=\bigcap_{j\ge 1}X_{j}\cap \left\{x\in X:\lim_{N\to\infty}\frac{1}{N}\sum_{n=1}^{N}f(T^{n}x)=\int_{X}fd\mu\right\}\cap
		\\ & \hspace{2cm}
		\bigcap_{j\ge 1}\left\{x\in X:\lim_{N\to\infty}\frac{1}{N}\sum_{n=1}^{N}|f-f_j|(T^{n}x)=\norm{f-f_j}_{L^{1}(\mu)}\right\}.
	\end{align*}
	By Birkhoff's ergodic theorem, we have $\mu(X_0)=1$.
	
	Fix $x\in X_0$.  Fix a uniquely ergodic topological dynamical system $(Y,S)$ with unique $S$-invariant Borel probability measure $\nu$, $g\in C(Y)$, and $y\in Y$.
	Then 
	\begin{align*}
		& \limsup_{N\to\infty}\left|\frac{1}{N}\sum_{n=1}^{N}f(T^{n}x)g(S^{\Omega(n)}y)-\int_{X}fd\mu\int_{Y}gd\nu\right|
		\\ \le &
		\limsup_{j\to\infty}\limsup_{N\to\infty}\left|\frac{1}{N}\sum_{n=1}^{N}(f-f_j)(T^{n}x)g(S^{\Omega(n)}y)\right|+\limsup_{j\to\infty}\left|\int_{X}(f-f_j)d\mu\int_{Y}gd\nu\right|
		\\ \le &
		2\norm{g}_{L^{\infty}(\nu)}\limsup_{j\to\infty}\norm{f-f_j}_{L^{1}(\mu)}
		\\ = &
		0.
	\end{align*}
	So, Theorem \ref{TA} can be reduced to $L^{\infty}$-functions.
	
	\medskip
	
	\noindent \textbf{Step III. Proving Theorem \ref{TA} for ergodic case and $L^{\infty}$-functions.}
	
	\medskip
	
	Fix an ergodic $(X,\X,\mu,T)$ and $1$-bounded $f\in L^{\infty}(\mu)$. Let $Z$ be the sub-$\sigma$-algebra of $\X$ generated by all eigenfunctions with respect to $T$. Then we can write $f$ as $\E_{\mu}(f|Z)+\left(f-\E_{\mu}(f|Z)\right)$.
	
	By repeating the argument in \textbf{Step II} and linear property of ergodic averages, we know that to prove that for $\E_{\mu}(f|Z)$, Theorem \ref{TA} holds, it suffices to show that Theorem \ref{TA} holds for all eigenfunctions with respect to $T$. Fix a non-constant eigenfunction $h$ with eigenvalue $\lambda$. Clearly, $\lambda\in \mathbb{T}, \lambda\neq 1$, and $\displaystyle \int_{X}hd\mu=0$. Fix $x\in \{x\in X:h(T^{n}x)=\lambda^{n} h(x)\ \text{for each}\ n\ge 1\}\cap \{x\in X:|h(x)|<\infty\}$.
	Fix a uniquely ergodic topological dynamical system $(Y,S)$ with unique $S$-invariant Borel probability measure $\nu$, $g\in C(Y)$, and $y\in Y$. Then $$\frac{1}{N}\sum_{n=1}^{N}h(T^{n}x)g(S^{\Omega(n)}y)=h(x)\frac{1}{N}\sum_{n=1}^{N}\lambda^{n}g(S^{\Omega(n)}y).$$
	By \cite[Corollary 1.25]{BR22}, we have 
	$$\lim_{N\to\infty}\frac{1}{N}\sum_{n=1}^{N}h(T^{n}x)g(S^{\Omega(n)}y)=\int_{X}hd\mu\int_{Y}gd\nu.$$ So, for such $h$, Theorem \ref{TA} holds. 
	
	When $h$ is a constant, by Theorem \ref{thm1-1}, we know that Theorem \ref{TA} holds. To sum up, Theorem \ref{TA} holds for $\E_{\mu}(f|Z)$.
	
	Next, we prove that Theorem \ref{TA} holds for $\tilde{f}:=f-\E_{\mu}(f|Z)$. Clearly, $\tilde{f}$  is orthogonal to the closed subspace of $L^{2}(\mu)$ spanned by all eigenfunctions with respect to $T$ and $\displaystyle \int_{X}\tilde{f}d\mu=0$. 
	
	Let 
	\begin{align*}
		& \bar{X}=\left\{x\in X: \lim_{N\to\infty}\frac{1}{N}\sum_{n=1}^{N}\tilde{f}(T^{n}x)=0\right\}\cap
		\\ & \hspace{3cm}
		 \bigcap_{p,q\in \mathbb{P},\atop p\neq q}\left\{x\in X: \lim_{N\to\infty}\frac{1}{N}\sum_{n=1}^{N}\tilde{f}(T^{pn}x)\bar{\tilde{f}}(T^{qn}x)=0\right\}.
	\end{align*}
	By Birkhoff's ergodic theorem and Theorem \ref{thm3-1}, $\mu(\bar{X})=1$.
	
	Fix $x\in \bar{X}$.  Fix a uniquely ergodic topological dynamical system $(Y,S)$ with unique $S$-invariant Borel probability measure $\nu$, $g\in C(Y)$, and $y\in Y$. Note that when $0\le g\le 1$, $g$ can be written as 
	$$\frac{1}{2}\left(\left(g+i\sqrt{1-g^2}\right)+\left(g-i\sqrt{1-g^2}\right)\right).$$ So, by linear property of ergodic averages, we can assume that $|g|\equiv 1$. 
	
	Fix two distinct prime numbers $p$ and $q$. Then 
	\begin{equation}\label{eq3-3}
		\begin{split}
		& \lim_{N\to\infty}\frac{1}{N}\sum_{n=1}^{N}\tilde{f}(T^{pn}x)g(S^{\Omega(pn)}y)\cdot \bar{\tilde{f}}(T^{qn}x)\bar{g}(S^{\Omega(qn)}y)
		\\ = &
		\lim_{N\to\infty}\frac{1}{N}\sum_{n=1}^{N}\tilde{f}(T^{pn}x) \bar{\tilde{f}}(T^{qn}x)\left(g\cdot \bar{g}\right)(S^{\Omega(n)+1}y)
		\\ = &
		\lim_{N\to\infty}\frac{1}{N}\sum_{n=1}^{N}\tilde{f}(T^{pn}x) \bar{\tilde{f}}(T^{qn}x)
		\\ = &
		0.
		\end{split}
	\end{equation}
	By combining \eqref{eq3-3} and Lemma \ref{lem2-1}, we know that 
	$$\lim_{N\to\infty}\frac{1}{N}\sum_{n=1}^{N}\tilde{f}(T^{n}x)g(S^{\Omega(n)}y)=\int_{X}\tilde{f}d\mu\int_{Y}gd\nu.$$
	So, Theorem \ref{TA} holds for $\tilde{f}$. This finishes the whole proof.
\end{proof}
\section{Applications of Theorem \ref{TA}}
Firstly, we prove Corollary \ref{corA}.
\begin{proof}[Proof of Corollary \ref{corA}]
	Note that for any bounded sequence $\{a_n\}_{n\ge 1}\subset \C$ and each $k\in \N$, we have 
	\begin{equation}\label{equ4-1}
		\lim_{N\to \infty}\left|\frac{1}{N}\sum_{n=1}^{N}a_{n}-\frac{1}{N}\sum_{n=1}^{N}a_{n+k}\right|=\lim_{N\to \infty}\left|\frac{1}{N}\sum_{n=1}^{N}a_{n}-\frac{1}{k}\sum_{i=0}^{k-1}\frac{1}{[N/k]}\sum_{n=1}^{[N/k]}a_{kn+i}\right|=0.
	\end{equation}
	Then we can assume that $\beta >0$. 
	
	If $\alpha\in \Q$, then there are $q,p\in \N$ such that for each $0\le i<q$, there are $d_i\ge 0,0\le c_i<p$ such that for any $n\in \N$, $[\alpha(qn+i)+\beta]=p(n+d_i)+c_i$. By combining \eqref{equ4-1} and \cite[Corollary 1.16]{BR22}, we have that Corollary \ref{corA} holds.
	
	Now, fix $\alpha\in \R\backslash\Q$. By \eqref{equ4-1}, we can assume that $\alpha>100$ and for each $n\in\N$, $\alpha n+\beta\notin \N$. For any $t\in \R$, let $\{t\}=t-[t]$.
	Note that $m\in\{[\alpha n+\beta]:n\in \N\}$ if and only if $$\left\{\frac{m-\beta}{\alpha}\right\}\in \left(1-\frac{1}{\alpha},1\right).$$ Then there exists an open interval $A_{\alpha,\beta}$ of $\mathbb{T}$ with length $\alpha^{-1}$ such that $m\in\{[\alpha n+\beta]:n\in \N\}$ if and only if $m/\alpha\in A_{\alpha,\beta}$. 
	
	Fix a uniquely ergodic topological dynamical system $(Y,S)$ with unique $S$-invariant Borel probability measure $\nu$, $g\in C(Y)$ with $0\le g\le 1$ and $y\in Y$. Let $T:\mathbb{T}\to \mathbb{T},x\mapsto x+\alpha^{-1}$. Then $(\mathbb{T},T)$ is uniquely ergodic and the unique $T$-invariant Borel probability measure is the Harr measure $\mu$ on $\mathbb{T}$. Let $\mathcal{B}({\mathbb{T}})$ be the Borel $\sigma$-algebra on $\mathbb{T}$. After applying Theorem \ref{TA} to $(\mathbb{T},\mathcal{B}({\mathbb{T}}),\mu,T)$ and $1_{A_{\alpha,\beta}}$, we have that for any $\ep\in (0,(100\alpha)^{-1})$, there is $x_{\ep}\in (0,\ep)$ such that 
	\begin{equation}\label{equ4-2}
		\lim_{N\to \infty}\frac{1}{N}\sum_{n=1}^{N}1_{A_{\alpha,\beta}}(x_{\ep}+n/\alpha)g(S^{\Omega(n)}y)=\frac{1}{\alpha}\int_{Y}gd\nu.
	\end{equation}
	By Weyl's uniformly distribution theorem, we have that 
	\begin{equation}\label{equ4-3}
		\lim_{N\to \infty}\frac{1}{N}\sum_{n=1}^{N}\left|1_{A_{\alpha,\beta}}(x_{\ep}+n/\alpha)-1_{A_{\alpha,\beta}}(n/\alpha)\right|\le 2\ep.
	\end{equation}
	Note that 
	\begin{equation}\label{equ4-4}
		\frac{1}{N}\sum_{n=1}^{N}g(S^{\Omega([\alpha n+\beta])}y)=\frac{[\alpha N+\beta]}{N}\cdot \frac{1}{[\alpha N+\beta]}\sum_{n=1}^{[\alpha N+\beta]}1_{A_{\alpha,\beta}}(n/\alpha)g(S^{\Omega(n)}y).
	\end{equation}
	By combining \eqref{equ4-2} - \eqref{equ4-4}, we know that 
	$$\lim_{N\to \infty}\frac{1}{N}\sum_{n=1}^{N}g(S^{\Omega([\alpha n+\beta])}y)=\int_{Y}gd\nu.$$ This finishes the proof.
\end{proof}
Next, we show Corollary \ref{corB}.
\begin{proof}[Proof of Corollary \ref{corB}]
	For any $\beta\in \R$, we define $T_{\beta}:\mathbb{T}^{k}\to\mathbb{T}^{k}$ by putting 
	$$T(x_1,\ldots,x_k)=(x_1+\beta,x_2+x_1,\ldots,x_{k}+x_{k-1})$$ for any $(x_1,\ldots,x_k)\in \mathbb{T}^{k}$. By \cite[Corollary 4.22]{EW11}, when $\beta$ is irrational, $(\mathbb{T}^{k},T_{\beta})$ is uniquely ergodic and the unique $T$-invariant Borel probability measure is $m^{\otimes k}$, where $m^{\otimes k}=m\times \cdots \times m(k\ \text{times})$.
	
	Let $$
	C=\begin{pmatrix}
	1 & 0 & \cdots & 0 \\
	\binom{1}{0} & \binom{1}{1} &\cdots & 0 \\
	\vdots & \vdots & \ddots & \vdots \\
	\binom{k}{0} & \binom{k}{1} & \cdots & \binom{k}{k}
	\end{pmatrix},
	B=\begin{pmatrix}
	1& 0 &\cdots & 0\\
	1 & 1 & \cdots & 1 \\
	\vdots & \vdots & \ddots & \vdots \\
	1 & k & \cdots & k^{k}
	\end{pmatrix}.
	$$
	Let $\pi_{k}:\mathbb{T}^{k}\to \mathbb{T}$ be the $k$-th coordinate projection. Then when 
	\begin{equation}\label{equ4-5}
		(c_0,\ldots,c_k)^{T}=B^{-1}C(x_k,\ldots,x_0)^{T},
	\end{equation}
	we have that $k!c_{k}=x_0$ and for each $n\in\N$,
	\begin{equation}\label{equ4-6}
		\pi_{k}\left(T_{x_0}^{n}(x_1,\ldots,x_k)\right)=(c_kn^k+\cdots+c_{1}n+c_0)\mod 1.
	\end{equation} 
	Fix $f\in L^{1}(m)$. Let $\mathcal{B}({\mathbb{T}^{k}})$ be the Borel $\sigma$-algebra on $\mathbb{T}^{k}$. After applying Theorem \ref{TA} to $(\mathbb{T}^{k},\mathcal{B}({\mathbb{T}^{k}}),m^{\otimes k},T_{k!\alpha})$ and $f\circ \pi_{k}$, we know that there is an $m^{\otimes k}$-null subset $X_{f}$ of $\mathbb{T}^{k}$ such that for any $(x_1,\ldots,x_k)\notin X_{f}$, any uniquely ergodic topological dynamical system $(Y,S)$ with unique $S$-invariant Borel probability measure $\nu$, any $g\in C(Y)$ and any $y\in Y$,
	$$\lim_{N\to\infty}\frac{1}{N}\sum_{n=1}^{N}f(\pi_{k}\left(T_{k!\alpha}^{n}(x_1,\ldots,x_k)\right))g(S^{\Omega(n)}y)=\int_{\mathbb{T}}fdm\int_{Y}gd\nu.$$
	
	\cite[Theorem 2.44.a]{Fo99} tells us that the image of any zero Lebesgue measure subset of $\R^{k}$ under an invertible linear map is still of zero Lebesgue measure. Combining this, \eqref{equ4-5} and \eqref{equ4-6}, we can find a set $A_f\subset \R^{k}$ with zero Lebesgue measure such that for any $(c_0,\ldots,c_{k-1})\notin A_{f}$, any uniquely ergodic topological dynamical system $(Y,S)$ with unique $S$-invariant Borel probability measure $\nu$, any $g\in C(Y)$ and any $y\in Y$,
	$$\lim_{N\to\infty}\frac{1}{N}\sum_{n=1}^{N}f(\alpha n^k+c_{k-1}n^{k-1}+\cdots+c_{1}n+c_0)g(S^{\Omega(n)}y)=\int_{\mathbb{T}}fdm\int_{Y}gd\nu.$$ This finishes the proof.
	\end{proof}
\section{Proofs of Theorem \ref{TB} and Theorem \ref{TC}}
Before proving Theorem \ref{TB} and Theorem \ref{TC}, let us introduce two results, which point out the characteristic factor behavior of the $\infty$-step factors.
\begin{thm}\label{thm4-1}
	$($\cite[Theorem 1]{HK05}, \cite[Theorem 3]{L05}$)$ Given $k\in \N$, let $P_1,\ldots,P_k\in \Z[n]$. Let $T$ be an invertible measure preserving transformation acting on Lebesgue space $(X,\X,\mu)$. Fix $f_1,\ldots,f_{k}\in L^{\infty}(\mu)$. If there is some $1\le j\le k$ such that $\E_{\mu}(f_j|\mathcal{Z}_{\infty}(T))=0$, then 
	$$\lim_{N\to\infty}\frac{1}{N}\sum_{n=1}^{N}\prod_{i=1}^{k}f_{i}(T^{P_{i}(n)}x)=0$$ in $L^{2}(\mu)$.
\end{thm}
\begin{thm}\label{thm4-2}
	$($\cite[Theorem 2.8]{FK22}$)$ Given $k\in \N$, let $T_1,\ldots,T_{k}$ be a family of commuting invertible measure preserving transformations acting on Lebesgue space $(X,\X,\mu)$. Let $P_1,\ldots,P_k\in \Z[n]$ and they are pairwise independent.  Fix $f_1,\ldots,f_{k}\in L^{\infty}(\mu)$. If there is some $j\in \{1,\ldots,k\}$ such that $\E_{\mu}(f_j|\mathcal{Z}_{\infty}(T_j))=0$, then
	$$\lim_{N\to\infty}\frac{1}{N}\sum_{n=1}^{N}\prod_{i=1}^{k}f_{i}(T_{i}^{P_{i}(n)}x)=0$$ in $L^{2}(\mu)$.
\end{thm}
Now, we begin to prove Theorem \ref{TB}.
\begin{proof}[Proof of Theorem \ref{TB}]
Without loss of generality, we can assume that $P_{1},\ldots,P_{k}$ have zero constant terms. Let $X\to \mathcal{M}(X),x\mapsto \mu_x$ be the ergodic decomposition of $\mu$ with respect to $S$. Fix $1$-bounded $g_1,\ldots,g_{k+1}\in L^{\infty}(\mu)$. By \cite[Theorem 1]{L05}, it suffices to prove the following:
\begin{equation}\label{eq4-5}
\begin{split}
& \lim_{N\to\infty}\frac{1}{N}\sum_{n=1}^{N}g_{1}(T_{1}^{P_{1}(n)}x)\cdots g_{k}(T_{k}^{P_{k}(n)}x)g_{k+1}(S^{\Omega(n)}x)
\\ & \hspace{2.5cm} =
\int_{X}g_{k+1}d\mu_{x}\lim_{N\to\infty}\frac{1}{N}\sum_{n=1}^{N}g_{1}(T_{1}^{P_{1}(n)}x)\cdots g_{k}(T_{k}^{P_{k}(n)}x)
\end{split}
\end{equation}
in $L^{2}(\mu)$

Firstly, we show that the characteristic factor of the ergodic averages stated in the left side of  \eqref{eq4-5} is $\mathcal{Z}_{\infty}(T_i),1\le i\le k$.  

By \cite[Proof of Lemma 2.15]{D24}, we can assume that $|g_{k+1}|\equiv 1$. Let the set $D=\{(p,q)\in \mathbb{P}^{2}:p\neq q\ \text{and there are}\ 1\le i\le j\le k,x,y\in \Z\backslash \{0\}\ \text{such that}\ xP_{i}(pn)+yP_{j}(qn)\equiv 0\}$. Then by the assumption on $P_1,\ldots,P_k$, a simple calculation gives that $D$ is empty or finite.

For each $n\in \N$, let $$A(n)=T_{1}^{P_{1}(n)}g_{1}\cdots T_{k}^{P_{k}(n)}g_{k}\cdot S^{\Omega(n)}g_{k+1}.$$
Fix two distinct $p,q\in \mathbb{P}$ such that $(p,q)\notin D$. Then 
\begin{align}
& \frac{1}{N}\sum_{n=1}^{N}\langle A(pn),A(qn)\rangle\notag
\\ = &
\frac{1}{N}\sum_{n=1}^{N}\int_{X}\prod_{i=1}^{k}g_{i}(T_{i}^{P_{i}(pn)}x)\cdot \prod_{i=1}^{k}\bar{g}_{i}(T_{i}^{P_{i}(qn)}x)\cdot \left(g_{k+1}\cdot \bar{g}_{k+1}\right)(S^{\Omega(n)+1}x)d\mu(x) \label{eq4-6}
\\ = &
\frac{1}{N}\sum_{n=1}^{N}\int_{X}\prod_{i=1}^{k}g_{i}(T_{i}^{P_{i}(pn)}x)\cdot \prod_{i=1}^{k}\bar{g}_{i}(T_{i}^{P_{i}(qn)}x)d\mu(x).\notag
\end{align}
Note that by the choices of $p$ and $q$, the polynomial family $$\{P_{1}(pn),\ldots,P_{k}(pn),P_{1}(qn),\ldots,P_{k}(qn)\}$$ is pairwise independent. By Theorem \ref{thm4-2} and \eqref{eq4-6}, we know that if there is some $1\le j\le k$ such that $\E_{\mu}(g_j|\mathcal{Z}_{\infty}(T_j))=0$, then
\begin{equation}\label{eq4-7}
\lim_{N\to\infty}\frac{1}{N}\sum_{n=1}^{N}\langle A(pn),A(qn)\rangle=0.
\end{equation}
Combining \eqref{eq4-7} and Lemma \ref{lem2-1}, we have that if there is some $1\le j\le k$ such that $\E_{\mu}(g_j|\mathcal{Z}_{\infty}(T_j))=0$, then 
\begin{equation}\label{eq4-8}
\lim_{N\to\infty}\frac{1}{N}\sum_{n=1}^{N}A(n)=0
\end{equation}
in $L^{2}(\mu)$. \eqref{eq4-8} means that 
\begin{equation}\label{eq4-9}
\begin{split}
& \lim_{N\to\infty}\Big|\frac{1}{N}\sum_{n=1}^{N}g_{1}(T_{1}^{P_{1}(n)}x)\cdots g_{k}(T_{k}^{P_{k}(n)}x)g_{k+1}(S^{\Omega(n)}x)- 
\\  & \hspace{0.5cm}
\frac{1}{N}\sum_{n=1}^{N}\E_{\mu}(g_1|\mathcal{Z}_{\infty}(T_1))(T_{1}^{P_{1}(n)}x)\cdots \E_{\mu}(g_k|\mathcal{Z}_{\infty}(T_k))(T_{k}^{P_{k}(n)}x)g_{k+1}(S^{\Omega(n)}x)\Big|=0
\end{split}
\end{equation}
in $L^{2}(\mu)$.

Next, we show that 
\begin{equation}\label{eq4-10}
\begin{split}
& \lim_{N\to\infty}\frac{1}{N}\sum_{n=1}^{N}\E_{\mu}(g_1|\mathcal{Z}_{\infty}(T_1))(T_{1}^{P_{1}(n)}x)\cdots \E_{\mu}(g_k|\mathcal{Z}_{\infty}(T_k))(T_{k}^{P_{k}(n)}x)g_{k+1}(S^{\Omega(n)}x)
\\ = &
\int_{X}g_{k+1}d\mu_{x}\lim_{N\to\infty}\frac{1}{N}\sum_{n=1}^{N}\E_{\mu}(g_1|\mathcal{Z}_{\infty}(T_1))(T_{1}^{P_{1}(n)}x)\cdots \E_{\mu}(g_k|\mathcal{Z}_{\infty}(T_k))(T_{k}^{P_{k}(n)}x)
\end{split}
\end{equation}
in $L^{2}(\mu)$. By combining Theorem \ref{thm4-2} and \eqref{eq4-9}, we know that if we prove this, then we finish the proof of \eqref{eq4-5}.

By \cite[Theorem 16.10]{HK-book}, for each $1\le i\le k$, there exists a sequence of functions $\{f_{i,j}\}_{j\ge 1}$ such that the following hold:
\begin{itemize}
	\item[(1)] for each $1\le i\le k,j\ge 1$ and $\mu$-a.e. $x\in X$, $\{f_{i,j}(T_{i}^{P_{i}(n)}x)\}_{n\in\Z}$ is a nilsequence;
	\item[(2)] for each $1\le i\le k,j\ge 1$, $\norm{\E_{\mu}(g_i|\mathcal{Z}_{\infty}(T_i))-f_{i,j}}_{L^{2k}(\mu)}\le 1/2^{j}$;
	\item[(3)] for each $1\le i\le k,j\ge 1$, $\norm{f_{i,j}}_{L^{\infty}(\mu)}\le 1$.
\end{itemize}
Then there exists $X_0\in \X$ with $\mu(X_0)=1$ such that the following hold:
\begin{itemize}
	\item[(1)] for any $y\in X_0$, each $1\le i\le k,j\ge 1$, and $\mu_y$-a.e. $x\in X$, $\{f_{i,j}(T_{i}^{P_{i}(n)}x)\}_{n\in\Z}$ is a nilsequence;
	\item[(2)] for any $y\in X_0$, $(X,\X,\mu_y,S)$ is ergodic;
	\item[(3)] for any $y\in X_0$, $\norm{g_{k+1}}_{L^{\infty}(\mu_y)}\le 1$;
	\item[(4)] for any $y\in X_0$ and each $1\le i\le k,j\ge 1$, $\norm{f_{i,j}}_{L^{\infty}(\mu_y)}\le 1$.
\end{itemize}
To prove \eqref{eq4-10}, it suffices to show that for any $y\in X_0$ and each $j\ge 1$, \begin{equation}\label{eq4-11}
\begin{split}
& \lim_{N\to\infty}\frac{1}{N}\sum_{n=1}^{N}f_{1,j}(T_{1}^{P_{1}(n)}x)\cdots f_{k,j}(T_{k}^{P_{k}(n)}x)g_{k+1}(S^{\Omega(n)}x)
\\ & \hspace{2cm} =
\int_{X}g_{k+1}d\mu_{y}\lim_{N\to\infty}\frac{1}{N}\sum_{n=1}^{N}f_{1,j}(T_{1}^{P_{1}(n)}x)\cdots f_{k,j}(T_{k}^{P_{k}(n)}x)
\end{split}
\end{equation} in $L^{2}(\mu_y)$.
To see this, let us do the following calculation:
\begin{align*}
& \limsup_{N\to\infty}\norm{\frac{1}{N}\sum_{n=1}^{N}\prod_{i=1}^{k}\E_{\mu}(g_i|\mathcal{Z}_{\infty}(T_i))(T_{i}^{P_{i}(n)}x)\cdot\left(g_{k+1}(S^{\Omega(n)}x)-\int_{X}g_{k+1}d\mu_{x}\right)}_{L^{2}(\mu)}
\\ \le &
\limsup_{j\to\infty}\limsup_{N\to\infty}\norm{\frac{1}{N}\sum_{n=1}^{N}\prod_{i=1}^{k}f_{i,j}(T_{i}^{P_{i}(n)}x)\cdot\left(g_{k+1}(S^{\Omega(n)}x)-\int_{X}g_{k+1}d\mu_{x}\right)}_{L^{2}(\mu)}+
\\ & \hspace{0.5cm}
2\limsup_{j\to\infty}\limsup_{N\to\infty}\left(\frac{1}{N}\sum_{n=1}^{N}\norm{\left(\prod_{i=1}^{k}T_{i}^{P_{i}(n)}f_{i,j}-\prod_{i=1}^{k}T_{i}^{P_{i}(n)}\E_{\mu}(g_i|\mathcal{Z}_{\infty}(T_i))\right)}_{L^{2}(\mu)}^{2}\right)^{1/2}
\\ \le &
\sup_{j\ge 1}\limsup_{N\to\infty}\left(\int_{X}\left|\frac{1}{N}\sum_{n=1}^{N}\prod_{i=1}^{k}f_{i,j}(T_{i}^{P_{i}(n)}x)\cdot\left(g_{k+1}(S^{\Omega(n)}x)-\int_{X}g_{k+1}d\mu_{x}\right)\right|^{2}d\mu(x)\right)^{1/2}
\\ \le &
\sup_{j\ge 1}\norm{\limsup_{N\to\infty}\norm{\left|\frac{1}{N}\sum_{n=1}^{N}\prod_{i=1}^{k}f_{i,j}(T_{i}^{P_{i}(n)}x)\cdot\left(g_{k+1}(S^{\Omega(n)}x)-\int_{X}g_{k+1}d\mu_{y}\right)\right|}_{L_{x}^{2}(\mu_y)}^{2}}_{L_{y}^{1}(\mu)}^{\frac{1}{2}}
\\ = &
0,
\end{align*}
where the last equality comes from \eqref{eq4-11}.

Now, we verify \eqref{eq4-11}. Fix $y\in X_0$ and $j\ge 1$. Then we get an ergodic measure preserving system $(X,\X,\mu_y,S)$. By \cite[Theorem 15.27]{G-book}, there exists a uniquely ergodic topological dynamical system $(Y,R)$ with unique $R$-invariant Borel probability measure $\nu$ such that $(X,\X,\mu_y,S)$ and $(Y,\B(Y),\nu,R)$ are isomorphic via the invertible measure preserving transformation $\pi:X\to Y$, where $\mathcal{B}(Y)$ is the Borel $\sigma$-algebra on $Y$. Then we can find 
$X_1\in \X$ with $\mu_{y}(X_1)=1$ and a sequence of functions $\{h_i\}_{i\ge 1}$ in $C(Y)$ such that the following hold:
\begin{itemize}
	\item[(1)] for each $i\ge 1$, $\norm{h_i\circ \pi - g_{k+1}}_{L^{2}(\mu_y)}\le 1/2^i$;
	\item[(2)] for any $x\in X_1,1\le t\le k$, $\{f_{t,j}(T_{t}^{P_{t}(n)}x)\}_{n\in\Z}$ is a nilsequence.
	\item[(3)] for any $x\in X_1$ and each $n\in\Z$, $\pi(S^{n}x)=R^{n}\pi(x)$;
	\item[(4)] for any $y\in Y$ and each $i\ge 1$, $|h_{i}(y)|\le 1$.
\end{itemize}
Note that the product of finitely many nilsequences is still a nilsequence. Then by Theorem \ref{thm2-1}, we know that for any $x\in X_1,i\ge 1$, 
\begin{equation}\label{eq4-12}
\begin{split}
& \lim_{N\to\infty}\frac{1}{N}\sum_{n=1}^{N}f_{1,j}(T_{1}^{P_{1}(n)}x)\cdots f_{k,j}(T_{k}^{P_{k}(n)}x)h_{i}\circ \pi(S^{\Omega(n)}x)
\\ = &
\lim_{N\to\infty}\frac{1}{N}\sum_{n=1}^{N}f_{1,j}(T_{1}^{P_{1}(n)}x)\cdots f_{k,j}(T_{k}^{P_{k}(n)}x)h_{i}(R^{\Omega(n)}\pi(x))
\\ = &
\int_{Y}h_id\nu\lim_{N\to\infty}\frac{1}{N}\sum_{n=1}^{N}f_{1,j}(T_{1}^{P_{1}(n)}x)\cdots f_{k,j}(T_{k}^{P_{k}(n)}x)
\\ = &
\int_{Y}h_i\circ \pi d\mu_{y}\lim_{N\to\infty}\frac{1}{N}\sum_{n=1}^{N}f_{1,j}(T_{1}^{P_{1}(n)}x)\cdots f_{k,j}(T_{k}^{P_{k}(n)}x).
\end{split}
\end{equation}
Based on \eqref{eq4-12}, by a standard approximation argument, we know that \eqref{eq4-11} exists in $L^{2}(\mu_{y})$. This finishes the whole proof.
\end{proof}
The proof of Theorem \ref{TC} is similar to one of Theorem \ref{TB}. The only difference between them is that we should use Theorem \ref{thm4-1} in the proof of Theorem \ref{TC} instead of Theorem \ref{thm4-2}.

Let $X\to \mathcal{M}(X),x\mapsto \mu_x$ be the ergodic decomposition of $\mu$ with respect to $S$. Once one finishes the proof of Theorem \ref{TC}, the following equality will appear:
\begin{equation}\label{eq4-13}
\begin{split}
& \lim_{N\to\infty}\frac{1}{N}\sum_{n=1}^{N}g_{1}(T^{P_{1}(n)}x)\cdots g_{k}(T^{P_{k}(n)}x)g_{k+1}(S^{\Omega(n)}x)
\\ & \hspace{2.5cm} =
\int_{X}g_{k+1}d\mu_{x}\lim_{N\to\infty}\frac{1}{N}\sum_{n=1}^{N}g_{1}(T^{P_{1}(n)}x)\cdots g_{k}(T^{P_{k}(n)}x)
\end{split}
\end{equation}
in $L^{2}(\mu)$, where $g_{1},\ldots,g_{k+1}\in L^{\infty}(\mu)$.

\section{Proofs of Proposition \ref{PA} - \ref{PC}}
Firstly, let us recall Furstenberg's corresponding principle.
\begin{thm}\label{thm5-1}
	$(${\em Cf.} \cite[Theorem 1.1]{F77}, \cite[Subsection 2.1]{FHB13}$)$ Let $k\in \N$ and $E\subset \Z^{k}$. Then exists a Lebesgue space $(X,\X,\mu)$,  commuting invertible measure preserving transformations $T_1,\ldots,T_k:X\to X$, and $A\in \X$ with $\mu(A)=d^{*}(E)$ such that $$d^{*}\left(E\cap (E-d_1)\cap \cdots \cap (E-d_m)\right)\ge \mu\left(A\cap \prod_{i=1}^{k}T_{i}^{-d_{1,i}}A\cdots \cap \prod_{i=1}^{k}T_{i}^{-d_{m,i}}A\right)$$ for all $m\in \N$ and all $(d_{1,1},\ldots,d_{1,k}),\ldots,(d_{m,1},\ldots,d_{m,k})\in \Z^{k}$.
\end{thm}
Based on Theorem \ref{thm5-1}, Proposition \ref{PA} can be deduced from the following result.
\begin{prop}\label{prop5-1}
	Let $S,T_1,\ldots,T_k$ be a family of invertible measure preserving transformations acting on Lebesgue space $(X,\X,\mu)$ and $T_1,\ldots,T_k$ are commuting. Let $P_1,\ldots,P_k$ be pairwise independent integer coefficients polynomials with zero constant terms and any one of them is not of the form $an^d$. Then for every $A\in\X$ with $\mu(A)>0$, there exists a constant $c$, depending only on $\mu(A)$ and $P_1,\ldots,P_k$, such that 
	$$\lim_{N\to\infty}\frac{1}{N}\sum_{n=1}^{N}\mu(A\cap T_{1}^{-P_{1}(n)}A\cap \cdots\cap T_{k}^{-P_{k}(n)}A\cap S^{-\Omega(n)}A)\ge c.$$
\end{prop}
Likely, Proposition \ref{PB} and Proposition \ref{PC} can be deduced from a similar result.
\begin{prop}\label{prop5-2}
	Let $S,T$ be two invertible measure preserving transformations acting on Lebesgue space $(X,\X,\mu)$. Let $P_{1},\ldots,P_{k}\in\Z[n]$ with zero constant terms. Then for every $A\in\X$ with $\mu(A)>0$, there exists a constant $c$, depending only on $\mu(A)$ and $P_1,\ldots,P_k$, such that 
	$$\lim_{N\to\infty}\frac{1}{N}\sum_{n=1}^{N}\mu(A\cap T^{-P_{1}(n)}A\cap \cdots\cap T^{-P_{k}(n)}A\cap S^{-\Omega(n)}A)\ge c.$$
\end{prop}
Before proving the above propositions, we introduce a quantitative version of the polynomial Szemer\'edi theorem.
\begin{thm}\label{thm5-2}
	$($\cite[Theorem 4.1]{FHB13}$)$ Let $T_1,\ldots,T_k$ be a family of commuting invertible measure preserving tranformations acting on Lebesgue space $(X,\X,\mu)$. Let $P_{1},\ldots,P_{k}\in\Z[n]$ with zero constant terms. Then for every $A\in\X$ with $\mu(A)>0$, there exists a constant $c$, depending only on $\mu(A)$ and $P_1,\ldots,P_k$, such that 
	$$\lim_{N\to\infty}\frac{1}{N}\sum_{n=1}^{N}\mu(A\cap T_{1}^{-P_{1}(n)}A\cap \cdots\cap T_{k}^{-P_{k}(n)}A)\ge c.$$
\end{thm}
Now, we are about to prove Proposition \ref{prop5-1}. 
\begin{proof}[Proof of Proposition \ref{prop5-1}]
	Fix $\delta\in (0,1)$. Fix a Lebesgue space $(X,\X,\mu)$, a family of commuting invertible measure preserving tranformations $T_1,\ldots,T_k$ acting on it, and $A\in \X$ with $\mu(A)=\delta$. Fix an invertible measure preserving tranformation $S$ acting on $(X,\X,\mu)$. By Theorem \ref{thm5-2}, there exists a constant $c(\delta)\in (0,1)$, depending on depending only on $\delta$ and $P_1,\ldots,P_k$, such that  
	\begin{equation}\label{eq4-1}
	\lim_{N\to\infty}\frac{1}{N}\sum_{n=1}^{N}\mu(A\cap T_{1}^{-P_{1}(n)}A\cap \cdots\cap T_{k}^{-P_{k}(n)}A)\ge c(\delta).
	\end{equation}
		Then by \eqref{eq4-1} and Theorem \ref{thm4-2}, we have 
		\begin{equation}\label{eq4-2}
		\lim_{N\to\infty}\frac{1}{N}\sum_{n=1}^{N}\int_{X}1_{A}(x)\cdot \prod_{j=1}^{k}\E_{\mu}(1_A|\mathcal{Z}_{\infty}(T_j))(T_{j}^{P_{j}(n)}x)d\mu(x)\ge c(\delta).
		\end{equation}
		By \cite[Theorem  16.10]{HK-book}, for each $1\le i\le k$, there exists a sequence of functions $\{\phi_{i,j}\}_{j\ge 1}$ such that the following hold:
		\begin{itemize}
			\item[(1)] for any $1\le i\le k,j\ge 1$ and $\mu$-a.e. $x\in X$, $\{\phi_{i,j}(T_{i}^{P_{i}(n)}x)\}_{n\in\Z}$ is a nilsequence;
			\item[(2)] for each $1\le i\le k,j\ge 1$, $\norm{\E_{\mu}(1_{A}|\mathcal{Z}_{\infty}(T_i))-\phi_{i,j}}_{L^{2k}(\mu)}\le 1/16^{jk}$;
			\item[(3)] for each $1\le i\le k,j\ge 1$, $0\le \phi_{i,j}\le 1$.
		\end{itemize}
		Then there exists a sufficiently large $j_0$ such that 
		\begin{equation}\label{eq6-a}
			\lim_{N\to\infty}\norm{\frac{1}{N}\sum_{n=1}^{N}\left(\prod_{i=1}^{k}\E_{\mu}(1_A|\mathcal{Z}_{\infty}(T_i))(T_{i}^{P_{i}(n)}x)- \prod_{i=1}^{k}\phi_{i,j_0}(T_{i}^{P_{i}(n)}x)\right)}_{L^{2}(\mu)}< c(\delta)^{3}/32.
		\end{equation}
		
%
%
		Let $X\to \mathcal{M}(X),y\mapsto \mu_y$ be the ergodic decomposition of $\mu$ with respect to $S$. 
		Note that the product of finitely many nilsequences is still a nilsequence. Then we can define $G:X\to [0,1]$ by putting 
		$$G(y)=\int_{X}1_{A}(x)\cdot \lim_{N\to\infty}\frac{1}{N}\sum_{n=1}^{N}\prod_{i=1}^{k}\phi_{i,j_0}(T_{i}^{P_{i}(n)}x)d\mu_{y}(x) $$ for $\mu$-a.e. $y\in X$. Clearly, 
		$$\int_{X}G(y)d\mu(y)=\lim_{N\to\infty}\frac{1}{N}\sum_{n=1}^{N}\int_{X}1_{A}(x)\cdot\prod_{i=1}^{k}\phi_{i,j_0}(T_{i}^{P_{i}(n)}x)d\mu(x).$$
		So, by \eqref{eq6-a} and \eqref{eq4-2}, we have $$\int_{X}G(y)d\mu(y)\ge 7c(\delta)/8.$$ 
		So,  
		\begin{align}
					& 7c(\delta)/8\le \int_{X}G(y)d\mu(y) \notag
							\\ = &
							\int_{\{y\in X:G(y)>c(\delta)/2\}}G(y)d\mu(y)+\int_{\{y\in X:G(y)\le c(\delta)/2\}}G(y)d\mu(y) \label{eq4-3}
							\\ \le &
							\mu(\{y\in X:G(y)>c(\delta)/2\})+c(\delta)/2. \notag
							\end{align}
							By \eqref{eq4-3}, we have 
							\begin{equation}\label{eq4-4}
							\mu(E:=\{y\in X:G(y)>c(\delta)/2\})\ge 3c(\delta)/8.
							\end{equation}
			
		By Theorem \ref{TB}, we know that the limit $$\lim_{N\to\infty}\frac{1}{N}\sum_{n=1}^{N}\mu(A\cap T_{1}^{-P_{1}(n)}A\cap \cdots\cap T_{k}^{-P_{k}(n)}A\cap S^{-\Omega(n)}A)$$ exists. 

		Next, we begin to estimate the uniform lower bound. 
		\begin{align*}
		& \lim_{N\to\infty}\frac{1}{N}\sum_{n=1}^{N}\mu(A\cap T_{1}^{-P_{1}(n)}A\cap \cdots\cap T_{k}^{-P_{k}(n)}A\cap S^{-\Omega(n)}A)
		\\ = &
		\lim_{N\to\infty}\int_{X}1_{A}(x)\cdot\left( \frac{1}{N}\sum_{n=1}^{N}1_{A}(T_{1}^{P_{1}(n)}x)\cdots 1_{A}(T_{k}^{P_{k}(n)}x)\cdot 1_{A}(S^{\Omega(n)}x)\right)d\mu(x)
		\\ = &
		\lim_{N\to\infty}\int_{X}1_{A}(x)\cdot \mu_{x}(A)\cdot\left( \frac{1}{N}\sum_{n=1}^{N}\prod_{j=1}^{k}1_{A}(T_{j}^{P_{j}(n)}x)\right)d\mu(x)\ (\eqref{eq4-5})
		\\ = &
		\lim_{N\to\infty}\int_{X}1_{A}(x)\cdot \mu_{x}(A)\cdot\left( \frac{1}{N}\sum_{n=1}^{N}\prod_{j=1}^{k}\E_{\mu}(1_A|\mathcal{Z}_{\infty}(T_j))(T_{j}^{P_{j}(n)}x)\right)d\mu(x)\ (\text{Theorem \ref{thm4-2}}) 
		\\ \ge &
	    \lim_{N\to\infty}\left|\frac{1}{N}\sum_{n=1}^{N}\int_{X}1_{A}(x)\cdot \mu_{x}(A)\cdot \prod_{i=1}^{k}\phi_{i,j_0}(T_{i}^{P_{i}(n)}x)d\mu(x)\right| -
	    \\ & \hspace{0.1cm}
	    \lim_{N\to\infty}\left|\int_{X}1_{A}(x)\cdot \mu_{x}(A)\cdot \frac{1}{N}\sum_{n=1}^{N}\left(\prod_{i=1}^{k}\E_{\mu}(1_A|\mathcal{Z}_{\infty}(T_i))(T_{i}^{P_{i}(n)}x)- \prod_{i=1}^{k}\phi_{i,j_0}(T_{i}^{P_{i}(n)}x)\right)d\mu(x)\right|
	    \\ \ge &	
	    \lim_{N\to\infty}\left|\frac{1}{N}\sum_{n=1}^{N}\int_{X}1_{A}(x)\cdot \mu_{x}(A)\cdot \prod_{i=1}^{k}\phi_{i,j_0}(T_{i}^{P_{i}(n)}x)d\mu(x)\right| - c(\delta)^{3}/32\ (\eqref{eq6-a})
	    \\ = &
	    \int_{X}\mu_{y}(A) \int_{X}1_{A}(x)\cdot \lim_{N\to\infty}\frac{1}{N}\sum_{n=1}^{N}\prod_{i=1}^{k}\phi_{i,j_0}(T_{i}^{P_{i}(n)}x)d\mu_{y}(x)d\mu(y)- c(\delta)^{3}/32
	    \\ = &
	    \int_{X}\mu_{y}(A)G(y)d\mu(y)- c(\delta)^{3}/32
	    \\ \ge &
	    \int_{E}G(y)^{2}d\mu(y)- c(\delta)^{3}/32\ (\text{by the fact that}\ G(y)\le \mu_{y}(A))
	    \\ \ge &
	    \frac{1}{16}c(\delta)^{3}.\ (\eqref{eq4-4})
		\end{align*}
		This finishes the whole proof.
\end{proof}
The proof of Proposition \ref{prop5-2} is similar to one of Proposition \ref{prop5-1}. The only difference is that we should use Theorem \ref{thm4-1}, Theorem \ref{TC} and \eqref{eq4-13} in the proof of Proposition \ref{prop5-2} instead of Theorem \ref{thm4-2}, Theorem \ref{TB} and \eqref{eq4-5}.
\section{Some questions}
\subsection{On Proposition \ref{PA}}
Due to those retrictions for polynomials in Theorem \ref{TB}, we can not give an answer to the following question:
\begin{ques}
	Fix $k\ge 2$ and $A\subset \N^{k+1}$ with positive upper Banach density. Are there $a\in \N^{k+1},d\in \N$ such that $$a,a+d\vec{e}_{1},\ldots,a+kd\vec{e}_{k},a+\Omega(d)\vec{e}_{k+1}\in A\ ?$$
\end{ques} 
So, we ask a special case of the above question here.
\begin{ques}\label{Q1}
	Fix $k\ge 2$. Is it true that for any finite coloring of $\N^{k+1}$, there are $a\in \N^{k+1},d\in \N$ such that the set $$\{a,a+d\vec{e}_{1},\ldots,a+kd\vec{e}_{k},a+\Omega(d)\vec{e}_{k+1}\}$$ is monochromatic ?
\end{ques}
\subsection{On recurrence times}
As a direct result of Theorem \ref{thm5-1} and Proposition \ref{prop5-2}, we know that for any $\delta>0$, there is $c(\delta)>0$, depending only on $\delta$, such that for any $E\subset \N$ with $d^{*}(E)=\delta$, then 
\begin{equation}\label{eq7-1}
	\liminf_{N\to \infty}\frac{1}{N}\sum_{n=1}^{N}d^{*}(E\cap (E-\Omega(n)))\ge c(\delta).
\end{equation}

In \cite{FHB07}, by combining some number theory results and a quantitative version of the Roth theorem, Frantzikinakis, Host and Kra proved that for any $E\subset \N$ with positive upper Banach density, there are infinitely many $n$ in $\mathbb{P}-1(\mathbb{P}+1)$ such that $$d^{*}(E\cap (E-n)\cap (E-2n))>0.$$ Later, Wooley and Ziegler \cite{WT12} extended it to general polynomials with zero constant terms.   
Based on these results and \eqref{eq7-1}, we expect similar result for $\Omega(n)$ here. That is,
\begin{ques}\label{Q2}
	Fix $E\subset \N$ with positive upper Banach density. Are there infinitely many $n$ in $\mathbb{P}-1(\mathbb{P}+1)$ such that
	$$\Omega(n)\in (E-E)\ ?$$
\end{ques}
If Question \ref{Q2} has a positive answer, then by letting $E$ be the arithmetic progressions with infinite length, we have that for each $k\in \N$, there are infinitely many $n$ in $\mathbb{P}-1(\mathbb{P}+1)$ such that $k|\Omega(n)$.

\bibliographystyle{plain}
\bibliography{ref}
\end{document}